\documentclass[a4paper,12pt]{amsart}
\usepackage{amsfonts,latexsym,amssymb,amscd,amsxtra}

\usepackage{ifthen}
\usepackage{graphicx}
\nonstopmode \numberwithin{equation}{section}

\usepackage[usenames]{color}
\newtheorem{thm}{Theorem}[section]
\newtheorem{lem}{Lemma}[section]
\newtheorem{cor}[thm]{Corollary}
\newtheorem{prop}[thm]{Proposition}

\newtheorem{step}{Step}[section]

\theoremstyle{definition}
\newtheorem{mlem}{Main lemma}[section]
\newtheorem{assertion}{Assertion}[section]
\newtheorem{cl}{Claim}[section]
\newtheorem{ca}{Case}[section]
\newtheorem{sca}{Subcase}[section]
\newtheorem{scl}{Subclaim}[section]
\newtheorem{conj}[thm]{Conjecture}
\newtheorem{fact}{Fact}[section]
\newtheorem{defn}[thm]{Definition}
\newtheorem{op}[thm]{Open Problem}

\newtheorem{rem}[thm]{Remark}
\newtheorem{example}[thm]{Example}

\numberwithin{equation}{section}

\newcounter {own}
\def\theown {\thesection       .\arabic{own}}

\newenvironment{pf}[1][]{%
 \vskip 3mm
 \noindent
 \ifthenelse{\equal{#1}{}}%
  {{\slshape Proof. }}%
  {{\slshape #1.} }%
 }%
{\qed\bigskip}

\newcounter{alphabet}
\newcounter{tmp}
\newenvironment{Thm}[1][]{\refstepcounter{alphabet}%
\bigskip%
\noindent%
{\bf Theorem \Alph{alphabet}}%
\ifthenelse{\equal{#1}{}}{}{ (#1)}%
{\bf .} \itshape}{\vskip 8pt}

\makeatletter
\newcommand{\RRef}[1]{\@ifundefined{r@#1}{}{\setcounter{tmp}{\ref{#1}}\Alph{tmp}}}
\makeatother

\newcounter{alphabet2}

\newcommand{\zb}{{\overline{z}}}

\newcommand{\D}{{\mathbb D}}
\newcommand{\T}{{\mathbb T}}

\newcommand{\im}{{\operatorname{Im}\,}}

\newcommand{\re}{{\operatorname{Re}\,}}

\newcommand{\arctanh}{{\operatorname{arctanh}}}

\def\be{\begin{equation}}
\def\ee{\end{equation}}

\newcommand{\ben}{\begin{enumerate}}
\newcommand{\een}{\end{enumerate}}

\newcommand{\blem}{\begin{lem}}
\newcommand{\elem}{\end{lem}}
\newcommand{\bthm}{\begin{thm}}
\newcommand{\ethm}{\end{thm}}
\newcommand{\bcor}{\begin{cor}}
\newcommand{\ecor}{\end{cor}}
\newcommand{\beg}{\begin{exam}}
\newcommand{\eeg}{\end{exam}}
\newcommand{\begs}{\begin{examples}}
\newcommand{\eegs}{\end{examples}}
\newcommand{\bdefe}{\begin{defn}}
\newcommand{\edefe}{\end{defn}}
\newcommand{\bques}{\begin{ques}}
\newcommand{\eques}{\end{ques}}
\newcommand{\bei}{\begin{itemize}}
\newcommand{\eei}{\end{itemize}}
\newcommand{\bcon}{\begin{conj}}
\newcommand{\econ}{\end{conj}}
\newcommand{\bop}{\begin{op}}
\newcommand{\eop}{\end{op}}

\newcommand{\bas}{\begin{assertion}}
\newcommand{\eas}{\end{assertion}}

\newcommand{\bfa}{\begin{fact}}
\newcommand{\efa}{\end{fact}}

\newcommand{\bca}{\begin{ca}}
\newcommand{\eca}{\end{ca}}

\newcommand{\bst}{\begin{step}}
\newcommand{\est}{\end{step}}

\newcommand{\bsca}{\begin{sca}}
\newcommand{\esca}{\end{sca}}

\newcommand{\bcl}{\begin{cl}}
\newcommand{\ecl}{\end{cl}}

\newcommand{\bmlem}{\begin{mlem}}
\newcommand{\emlem}{\end{mlem}}

\newcommand{\bscl}{\begin{scl}}
\newcommand{\escl}{\end{scl}}

\newcommand{\bcons}{\begin{conjs}}
\newcommand{\econs}{\end{conjs}}

\newcommand{\bprop}{\begin{prop}}
\newcommand{\eprop}{\end{prop}}

\newcommand{\br}{\begin{rem}}
\newcommand{\er}{\end{rem}}
\newcommand{\brs}{\begin{rems}}
\newcommand{\ers}{\end{rems}}
\newcommand{\bo}{\begin{obser}}
\newcommand{\eo}{\end{obser}}
\newcommand{\bos}{\begin{obsers}}
\newcommand{\eos}{\end{obsers}}
\newcommand{\bpf}{\begin{pf}}
\newcommand{\epf}{\end{pf}}
\newcommand{\ba}{\begin{array}}
\newcommand{\ea}{\end{array}}
\newcommand{\beq}{\begin{eqnarray}}
\newcommand{\beqq}{\begin{eqnarray*}}
\newcommand{\eeq}{\end{eqnarray}}
\newcommand{\eeqq}{\end{eqnarray*}}

\newcounter{minutes}
\divide\time by 60
\newcounter{hours}
\multiply\time by 60 \addtocounter{minutes}{-\time}

\begin{document}

\bibliographystyle{amsplain}
\title []
{Estimates of partial derivatives for harmonic functions on the unit disc}

\def\thefootnote{}
\footnotetext{ \texttt{\tiny  \number\day-\number\month-\number\year
          }
} \makeatletter\def\thefootnote{\@arabic\c@footnote}\makeatother

\subjclass[2010]{Primary: 30C62, 31A05; Secondary: 30H10, 30H20.}
 \keywords{Poisson integral,  Hardy  space, Bergman  space, Harmonic conjugate}

\author{Adel Khalfallah}
\address{Department of Mathematics, King Fahd University of Petroleum and
	Minerals, Dhahran 31261, Saudi Arabia}\email{khelifa@kfupm.edu.sa}

\author{ Miodrag Mateljevi\'c}
\address{M. Mateljevi\'c, Faculty of mathematics, University of Belgrade, Studentski Trg 16, Belgrade, Republic of Serbia}
\email{miodrag@matf.bg.ac.rs}
\maketitle
\begin{abstract}
Let $f = P[F]$ denote the Poisson integral of $F$ in the unit disk $\D$ with $F$ is an absolute continuous in the unit circle $\T$ and $\dot{F}\in  L^p(\T)$, where $\dot{F}(e^{it}) = \frac{d}{dt} F(e^{it})$
and $p \in [1,\infty]$. Recently,    Chen et al. \cite{CPW} (J. Geom. Anal., 2021) extended Zhu's results \cite{Zhu} (J. Geom. Anal., 2020) and  proved that (i) if  $f$ is a
harmonic mapping and $1 \leq p < \infty$, then $f_z$ and $\overline{f_{\bar{z}}} \in B^p(\D)$, the  Bergman
spaces of $\D$.   Moreover,  (ii) under additional conditions  as  $f$ being harmonic quasiregular mapping in \cite{Zhu} or $f$  being  harmonic elliptic mapping in \cite{CPW},  they proved that  $f_z$ and $\overline{f_{\bar{z}}}\in  H^p(\D)$, the  Hardy space of $\D$, for $1 \leq p \leq \infty$. The aim of this paper is to extend these  results by showing that (ii) holds for $p\in(1,\infty)$ without any extra conditions and for $p=1$ or $p=\infty$,   $f_z$ and $\overline{f_{\bar{z}}}\in  H^p(\D)$ if and only if $H(\dot{F})\in L^p(T)$, the Hilbert transform of $\dot{F}$ and in that case, it yields  $zf_z=P[\frac{\dot{F}+iH(\dot{F})}{2i}]$.

\end{abstract}

\maketitle \pagestyle{myheadings}
\markboth{ A. Khalfallah  and M Mateljevi\'c}{}

\section{Preliminaries}

We denote by $\D$ 
  the unit disk and   $\T:=\partial\D$  the
unit circle.
For $z=x+iy\in\mathbb{C}$, the two complex
differential operators are defined by

$$\partial_z=\frac{1}{2}\left(\frac{\partial}{\partial x}-i\frac{\partial}{\partial y}\right)~\mbox{and}~
\overline{\partial}_z=\frac{1}{2}\left(\frac{\partial}{\partial x}+i\frac{\partial}{\partial y}\right).$$

Also, we denote by  $h(\D)$ (resp. $H(\D)$ the set of harmonic (resp. analytic) functions on $\D$.

\subsection{Hardy and Bergman spaces}\hfill

For $p\in[1,\infty]$, the {\it harmonic Hardy space
${h}^{p}(\mathbb{D})$} consists of all
harmonic   functions from $\mathbb{D}$ to $\mathbb{C}$ such that
$M_{p}(r,f)$ exists for all $r\in(0,1)$, and $ \|f\|_{p}<\infty$,
where
$$M_{p}(r,f)=\left(\frac{1}{2\pi}\int_{0}^{2\pi}|f(re^{i\theta})|^{p}\,d\theta\right)^{\frac{1}{p}}
$$
and
$$\|f\|_{p}=
\begin{cases}
\displaystyle\sup \{M_{p}(r,f):\; 0<r <1\}
& \mbox{if } p\in[1,\infty),\\
\displaystyle\sup\{|f(z)|:\; z\in\mathbb{D}\} &\mbox{if } p=\infty.
\end{cases}$$

The  {\it analytic Hardy space} $H^{p}(\mathbb{D})$ is the set of  all  elements of $h^{p}(\D)$ which are analytic. (cf. \cite{Du,Du1}).\\

Denote by $L^{p}(\mathbb{T})~(p\in[1,\infty])$ the space of all measurable functions  $F$
of $\mathbb{T}$ into $\mathbb{C}$ with

$$\|F\|_{L^{p}}=
\begin{cases}
\displaystyle\left(\frac{1}{2\pi}\int_{0}^{2\pi}|F(e^{i\theta})|^{p}d\theta\right)^{\frac{1}{p}}
& \mbox{if } p\in[1,\infty),\\
\displaystyle \sup\{|F(e^{i\theta})|:\; \theta\in [0, 2\pi)\} &\mbox{if } p=\infty.
\end{cases}
$$

For  $z\in\mathbb{D}$, let
$$P(z)=\frac{1-|z|^{2}}{|1-z|^{2}}$$ be the {\it Poisson kernel}.
For a mapping $F\in L^{1}(\mathbb{T})$, the {\it Poisson integral} of $F$
is defined by

$$f(z)=P[F](z)=\frac{1}{2\pi}\int_{0}^{2\pi}P(ze^{-i\theta})F(e^{i\theta})d\theta.$$

It is well known that a  function $f$ belongs to $h^p(\D)$ ($1 < p \leq \infty$) if and only if $f=P[\phi]$, with  $\phi \in  L^p(\T)$. And if $f = P[\phi]$, $\phi \in  L^p(\T)$, then $$\|f\|_p=\|\phi\|_p, \quad (1\leq p \leq \infty),$$ and 
$$f_*(e^{i\theta}):=\lim_{r\to 1^-} f(re^{i\theta})=\phi(e^{i\theta})  \quad a.e.$$
Thus for $p\in (1,\infty]$, $P: L^{p}(\T) \to h^p(\D)$ is an isometry, see \cite{Du,Gar,Koos}.\\

For $p\in[1,\infty]$, the {\it Bergman space
$B^{p}(\mathbb{D})$} consists of all analytic functions $f:\;\mathbb{D}\rightarrow\mathbb{C}$ such that
$$\|f\|_{b^{p}}=
\begin{cases}
\displaystyle\left(\int_{\mathbb{D}}|f(z)|^{p}d\sigma(z)\right)^{\frac{1}{p}}
& \mbox{if } p\in[1,\infty),\\
\displaystyle \sup\{|f(z)|:\; z\in \mathbb{D}\} &\mbox{if } p=\infty,
\end{cases}
$$
where  $d\sigma(z)=\frac{1}{\pi}dxdy$ denotes  the normalized Lebesgue area measure on $\D$. Obviously, $H^{p}(\D)\subset B^{p}(\D)$ for each $p\in [1,\infty]$,  (cf. \cite{HKZ}).\\

\subsection{Harmonic conjugates and M. Riesz theorem}
If $f$ is a harmonic  function  on $\D$  given by 
$$f(re^{i\theta})= \sum_{n=-\infty}^{\infty} c_n r^{|n|} e^{in\theta},
$$
then its harmonic conjugate is defined by 
$$\tilde{f}(re^{i\theta}):= \sum_{n=-\infty}^{\infty} m_n c_n r^{|n|} e^{in\theta},$$ 
where $m_n=-i\, {\rm sign}\, n$; in particular $\tilde{f}(0)=0$. In addition, we have  
\be \label{cong}
f(z)+i \tilde{f}(z)=-f(0) + 2 \sum_{n=0}^{\infty} c_n z^n.
\ee
The expression $\sum_{n\geq 0} c_n z^n$ is called the {\it Riesz projection} of $f$, and denoted by $P_+f$. i.e.,
$$P_+ : h(\D) \to H(\D) $$ is  defined by $$P_+(f)(z):=\sum_{n\geq 0} c_n z^n.$$
Thus Eq. (\ref{cong}) can be rewritten as 
\be\label{harconj2}
f+i\tilde{f}=-f(0)+2P_{+}(f). 
\ee
\noindent Next, assume that $f=P[\phi], \phi \in L^1(\T)$, the function conjugate to $P[\phi]$, equals

$$\tilde{P}[\phi](z)= \frac{1}{2\pi}\int_{0}^{2\pi} \tilde{P}(ze^{-i\theta})\phi(e^{i\theta})d\theta, $$

\noindent here $\tilde{P}$ denotes the conjugate Poisson kernel,
 $$\tilde{P}(z)=\im \frac{1+z}{1-z}=\frac{2r \sin \theta}{1+r^2-2r \cos(\theta)}, \quad z=re^{i\theta}.$$
 The Riesz projection $P_+$ may also be written as a Cauchy type integral,
 \be 
 P_+(P[\phi])(z)=C(\phi)(z):= \frac{1}{2\pi i} \int_{\T} \frac{\phi(w)}{w-z}\, dw, \quad z\in \D.
 \ee

Remark that in the case $f=P[\phi]$, the Eq. (\ref{harconj2}) can be rewritten as
\be\label{cong:2}
P[\phi]+i \tilde{P}[\phi]=P[\phi](0)+2P_+(\phi). 
\ee

Moreover, if $\phi \in L^1(\T)$, then $\tilde{P}[\phi]$ has radial limits almost everywhere and there holds the relation
$$
\lim_{r\to 1^{-}} \tilde{P}[\phi](re^{i\theta})=\tilde{\phi}(e^{i\theta}), 
$$

\noindent where $\tilde{\phi}$ is the {\it Hilbert transform} of $\phi$ given by 
$$ H(\phi)(e^{i\theta})=\tilde{\phi}(e^{i\theta}):= \frac{1}{2\pi} \,{ \rm p.v. } \int_{-\pi}^{\pi} \frac{\phi(e^{i\theta})}{\tan(\frac{\theta-t}{2})} \, dt. $$

The Hilbert operator maps 
$L^1(\T)$ into $L^p(\T)$, for every $p <1$, but not into $L^1(\T)$, so in the general case the Poisson integral of $\tilde{\phi}$ has no sense. However we have 

\begin{Thm}{\cite{Mash}}
If $\phi\in L^1(\T)$ and $\tilde{\phi}\in L^1(\T)$, then 
$$P[\tilde{\phi}]=\tilde{P}[\phi]. $$
\end{Thm}

It is natural to consider the following question: if $f$ is harmonic in $\D$ and $
 f \in h^p(\D)$, $1\le p \le \infty$, does $\tilde{f}\in h^p(\D)$ (or $P_+(f) \in H^p(\D)$)?\\
The question has an affirmative answer for all $1<p<\infty$. This is the content of the famous theorem proved by M. Riesz \cite{Riesz}, see also Rudin \cite[Theorem 17.26]{rud1}.

\begin{thm}(M. Riesz)  For $1<p<\infty$,  the Riesz projection $P_+$ 
$$P_+: L^p(\T) \to H^p(\D) $$
is a continuous operator on $L^p(\T)$, i.e.,
there is a constant $A_p$ such that 
\be\label{ries:ineq}
\|P_+(f)\|_p \le A_p \| f \|_p
\ee
holds for $\ f \in L^p(\T)$.
\end{thm}

Equivalently, for every $p\in (1,\infty)$ there exists a constant $B_p$ such that
$$\|u+i\tilde{u}\|_p \leq B_p\|u\|_p, $$
holds for any harmonic function   $u\in h^p(\D)$, and  $\tilde{u}$ is its harmonic conjugate  normalized by  $\tilde{u}(0)=0$.\\

Recently Hollenbeck and Verbitsky \cite{ver} proved that the best possible constant in the Riesz's inequality (\ref{ries:ineq}) is $A_p=\csc(\frac{\pi}{p})$.\\

It turns out that the Riesz theorem  is not true  for  $ p = 1$ and $p = \infty$. Indeed, if
$$f_1(z)=\frac{1+z}{1-z}. \quad f_2(z)=i\log(1-z), \quad z\in \D.
$$
Then $\re f_1=P$, the Poisson kernel, so $\re f_1 \in h^1,$ and $\re f_2 \in h^\infty$. However $f_1 \not \in H^1$, because its boundary function is not integrable, and $f_2$ is obviously unbounded. \\

As a consequence of M. Riesz's theorem, we deduce that the Hilbert transform sends $L^p(\T)$ into $L^p(\T)$ for $p\in(1,\infty)$.

\section{Main results}

In \cite{Zhu}, Zhu considered the following question :\\

Under what conditions on the boundary function $F$ ensure that the
partial derivatives of its harmonic extension $f=P[F]$, i.e., $f_z$ and $\overline{f_{\overline{z}}}$, are in the space $\mathcal{B}^{p}(\mathbb{D})$ (or $H^p(\D)$), where $p\in  [1,\infty]$?\\

 The author  proved Theorem B only for $p\in[1,2)$ and then extended  to $[1,\infty)$ by Chen et al. in \cite{CPW}.

\begin{Thm}{\rm (\cite[Theorem 1.2]{Zhu}),\cite[Theorem 1.1]{CPW})}\label{Zhu-1}
Suppose that $f=P[F]$ is a harmonic mapping in $\mathbb{D}$ and
$\dot{F}\in L^{p}(\mathbb{T})$, where $F$ is an  absolutely continuous  function. \ben
\item
If $p\in[1,\infty)$, then both $f_{z}$ and $\overline{f_{\overline{z}}}$ are in $\mathcal{B}^{p}(\mathbb{D}).$
\item
If $p=\infty$, then there exists a harmonic
mapping $f=P[F]$, with  $\dot{F}\in L^{\infty}(\mathbb{T})$, such that
neither $f_{z}$ nor $\overline{f_{\overline{z}}}$ is in $\mathcal{B}^{\infty}(\mathbb{D}).$
\een
\end{Thm}

Furthermore, under some additional conditions of $f$, they proved that the  partial derivatives are in $H^p(\D)$, for $p\in[1,\infty]$.

\begin{Thm}{\rm (\cite[Theorem 1.3]{Zhu},\cite[Theorem 1.2]{CPW})}\label{Zhu-2}
Suppose that $p\in[1,\infty]$ and $f=P[F]$ is a $(K,K')$-elliptic mapping in $\mathbb{D}$ with
$\dot{F}\in L^{p}(\mathbb{T})$, where $F$ is an absolute continuous function, $K\geq1$ and $K'\geq0$. Then both
$f_{z}$ and $\overline{f_{\overline{z}}}$ are in $H^{p}(\mathbb{D}).$
\end{Thm}

In \cite{CPW}, the authors showed  that Theorem C also holds true for harmonic elliptic mappings, which are more general than harmonic quasiregular mapping.\\

\noindent Our first main result is a refinement of the two previous theorems, we prove that for $1<p<\infty$, both
$f_{z}$ and $\overline{f_{\overline{z}}}$ are in $H^{p}(\mathbb{D})$ without any extra conditions on 
$f$.

\begin{thm}\label{main1}
Suppose that $F$ is an  absolute continuous  function on $\T$ and  $f=P[F]$ is a harmonic mapping in $\mathbb{D}$ and
$\dot{F}\in L^{p}(\mathbb{T})$.
\ben
\item
If $p\in(1,\infty)$, then both $f_{z}$ and $\overline{f_{\overline{z}}}$ are in $\mathcal{H}^{p}(\mathbb{D}).$ Moreover, there exits a constant $A_p$ such that
$$ \max(\|f_z \|_p, \|\overline{f_{\overline{z}}} \|_p) \leq A_p \| \dot{F}\|_p .$$

\item If $p=1$ or $p=\infty$, then both  
  $f_{z} $ and  $\overline{f_{\overline{z}}}$ are in $\mathcal{H}^p(\mathbb{D})$  if and only if  $H(\dot{F}) \in L^p(\T)$. Moreover
  $$2izf_z(z)=P[\dot{F}+iH(\dot{F})](z).$$
\een
\end{thm}

Let  $z=re^{i\theta}\in {\mathbb {D}}$. The polar derivatives of 
a complex valued function $f$ are given as follows 
$$
\begin{aligned} \left\{ \begin{array}{lr} \displaystyle f_{\theta}(z):=\frac{\partial f(z)}{\partial \theta}=i\big (zf_{z}(z)-\overline{z}f_{\overline{z}}(z)\big )\\ \displaystyle f_{r}(z):=\frac{\partial f(z)}{\partial r}=f_{z}(z)e^{i\theta}+f_{\overline{z}}(z)e^{-i\theta}. \end{array} \right. \end{aligned}
$$

Thus, if $f$ is harmonic, then  $f_\theta$ and $rf_r$ are harmonic too and 
\be\label{rtheta}
2izf_z=f_\theta +irf_r. 
\ee

Here, we should mention  the following fundamental observation that $rf_r$ is a harmonic conjugate to $f_\theta$.\\

 The next lemma  shows that $f_\theta \in h^p(\T)$ for $f=P[F]$ and $\dot{F}\in L^p(\T)$, for $p\in [1,\infty]$.
\begin{lem}{\cite[Lemma 2.3]{Zhu}}\label{zhu:lem}
Suppose $1\leq p\leq \infty$, $f=P[F]$ is a harmonic mapping of $\D$ with an absolutely continuous boundary function $F$  satisfying $\dot{F}\in L^p(\T)$. Then
$$
f_\theta=P[\dot{F}]. 
$$

In particular, $f_\theta\in h^p(\D)$ and 
$\|f_\theta\|_p \leq \| \dot{F}\|_{L^p}$.

\end{lem}

As a corollary, we obtain some information regarding the harmonic function $rf_r$, which is a harmonic conjugate to $f_\theta$.

\begin{cor}\label{main2}
Suppose that $F$ is an  absolutely continuous  function and  $f=P[F]$ is a harmonic mapping in $\mathbb{D}$ and
$\dot{F}\in L^{p}(\mathbb{T})$.
\ben
\item
If $p\in(1,\infty)$, then $rf_r \in h^{p}(\mathbb{D}).$
\item If $p=1$ or $p=\infty$, then $rf_r \in h^p(\mathbb{D})$ if and only if  $H(\dot{F}) \in L^p(\T)$.
\een
\end{cor}

\section{Proofs}

The next result provides an integral representation of $f_z$ and $f_{\bar{z}}$ in terms of the Riesz projection of $\dot{F}$. 
\begin{prop}\label{prop:integ}
	Let $f=P[F]$ where $F$ is an  absolute continuous function on $\T$.
	 Then
	$$
	f_z= \frac{1}{2\pi i }\int_\T \frac{\dot{F}(w)}{w-z} \, |dw|, 
	$$
	and 
	$$
	f_{\bar {z}}= \frac{1}{2\pi i }\int_\T \frac{\dot{F}(w)}{\bar{w}-\bar{z}} \, |dw|. $$
	Thus
	\be\label{partialf}
	izf_z=P_+[\dot{F}](z), \quad  and \quad i z \overline{f_{\bar{z}}}(z)=P_+[\overline{\dot{F}}](z).  
	\ee
\end{prop}

\begin{proof}
Let $P$ be the Poisson kernel. Then
	$$P_z=\partial P(z)=\frac{1}{(1-z)^2}.$$
	Thus 
	$$f_z(z)= \frac{1}{2\pi} \int_\T \partial P(ze^{-it}) e^{-it} F(e^{it}) dt =\frac{1}{2\pi}\int_\T \frac{e^{-it}}{(1-ze^{-it})^2} F(e^{it}) dt. $$
	As
	$$\frac{d}{dt}\left( \frac{1}{1-ze^{-it}}\right)= \frac{-ize^{-it}}{(1-ze^{-it})^2},$$
	
	\noindent and $F$ is absolutely continuous, using integration by parts, we obtain 
	$$ iz  f_z(z)=\frac{1}{2\pi} \int_\T \frac{\dot{F}(e^{it})}{1-ze^{-it}} dt=\frac{1}{2\pi i} \int_{\T} \frac{\dot{F}(w)}{w-z}\, dw=C[\dot{F}]=P_+[\dot{F}].$$
	Hence, 
	$$	f_z(z)= \frac{1}{2\pi i }\int_\T \frac{e^{-it}}{1-ze^{-it}} \dot{F}(e^{it})\, dt=\frac{1}{2\pi i }\int_\T \frac{\dot{F}(w)}{w-z} \, |dw|.$$
Similar computations provide the expression of $f_{\bar{z}}$.
\end{proof}

\noindent {\it Proof of Theorem \ref{main1}.} \hfill

(1) Let $p\in (1,\infty)$ and $f=P[F]$ with $\dot{F} \in L^p(\T)$. Combining  Proposition \ref{partialf} and   Riesz's theorem, we conclude that  $P_+[\dot{F}]=izf_z$  is an element of $H^p$   and there exists a constant $A_p$ such that
$\|zf_z \|_{p} \leq A_p \|\dot{F}\|_{L^p}.$
Since $\|zf_z \|_{p}=\|f_z \|_{p}$, we deduce that $f_z \in H^p$ and $\|f_z \|_{p} \leq A_p \|\dot{F}\|_{L^p}.$

(2) Let $p=1$ or $p=\infty$. Let $h=P[\dot{F}]$. Clearly $h$ is harmonic and $h(0)=0$. Thus by (\ref{harconj2}) we get 
$h+i\tilde{h}=2P_{+}(h)$, that is
\be\label{eq:fund}
P[\dot{F}]+i \tilde{P}[\dot{F}]=2P_+(\dot{F})=2izf_z.
\ee
Assume that $f_z\in H^p(\D)$, then its boundary $(f_z)_*$ belongs to $L^p(\T)$ obtained by taking radial limits, that is,  $(f_z)_*(\xi)=\lim_{r\to 1} f_z(r\xi)$ a.e. As $h_*=\dot{F}$, by (\ref{eq:fund}), we deduce that $(\tilde{h})_*= (\tilde{P}[\dot{F}])_*=H(\dot{F})\in L^p(\T)$.

Conversely, assume that $H(\dot{F})\in L^{p}(\T)$, then 
$
\tilde{P}[\dot{F}]=P[H(\dot{F})].$ Therefore, 
$2izf_z=P[\dot{F}+i H(\dot{F}) ]$  belongs to $H^p(\D)$, as $\dot{F}+i H(\dot{F}) \in L^p(\T)$. \hfill $\square$\\

\noindent {\it Proof of Corollary \ref{main2}.} \hfill

1) Let $p\in(1,\infty)$. As $rf_r$ is the harmonic conjugate of $f_\theta=P[\dot{F}]\in h^p(\D)$, by Riesz theorem, we deduce that $rf_r \in h^{p}(\D)$.

2) In the case $p=1$ or $p=\infty$. Using the identity 
$$irf_r(z)=2izf_z(z)-P[\dot{F}](z),
$$
 we deduce that $rf_r \in h^p$ if and only if $f_z\in H^p$, since $P[\dot{F}]\in h^p(\D)$. The conclusion follows from Theorem \ref{main1}. \\

\begin{example}
We provide an example of a harmonic mapping $f=P[F]$, where $F$ is an absolutely continuous function with $\dot{F}\in L^\infty(\T)$, such that neither $f_z$ nor $f_{\overline z}$ is in $L^\infty(\D)$.\\

Let $F(t):= |t|$ on $[-\pi,\pi]$. $F$ is an absolute continuous function on $\T$ and 
$\dot{F}=-1$ on $(-\pi,0)$ and $\dot{F}=1$ on $(0,\pi)$. By Proposition \ref{prop:integ}, $f_z$ has the following integral representation    $$
 iz  f_z(z)=\frac{1}{2\pi} \int_\T \frac{\dot{F}(e^{it})}{1-ze^{-it}} dt=\sum_{n\geq 0} \widehat{\dot{F}}(n) z^n, 
	$$
	where $\widehat{\dot{F}}(n)$ denotes the Fourier coefficient of $\dot{F}$ of order $n$. 
Elementary computations guarantee that for $z\in \D$, 

$$ izf_z(z)=-\frac{2i}{\pi} \sum_{n\geq 0} \frac{z^{2n+1}}{2n+1}=- \frac{2i}{\pi}\arctanh (z)=- \frac{2i}{\pi} \log \left(\frac{1+z}{1-z} \right).$$

Thus $|f_z(r)|$ goes to $\infty$ as $r\to 1$. In addition, one  can compute the Hilbert transform of $\dot{F}$. Indeed,  the boundary function of $izf_z$ is given by $- \frac{2i}{\pi} \log \left(\frac{1+e^{i\theta}}{1-e^{i\theta}} \right)=-\frac{2i}{\pi} \log(i\cot  \frac{\theta}{2}  )=\dot{F}+iH(\dot{F}),$ and $izf_z=P[\dot{F}+iH(\dot{F})]$, as $f_z\in H^2$.
This shows that
$$H(\dot{F})=\frac{2}{\pi}  \ln | \tan \frac{\theta}{2} |\not \in L^\infty(\T).$$
\hfill$\square$
\end{example}

We close this paper by providing an elementary proof of Theorem B using only Proposition \ref{prop:integ} and without switching to polar coordinates as in \cite{Zhu,CPW}.\\

Let $p\in [1,\infty)$ and $f=P[F]$, where $\dot{F}\in L^p(\T)$, we will prove that $f_z$ is in $L^p(\D)$. By Proposition \ref{prop:integ}, $f_z$ has the following integral representation    $$
	f_z(z)= \frac{1}{2\pi i }\int_\T \frac{\dot{F}(w)}{w-z} \, |dw|. 
	$$
According to Jensen's inequality, we have

$$|f_z(z)|^p \leq \frac{1}{2\pi} \int_0^{2\pi} \frac{|\dot{F}(e^{it})|^p}{|1-ze^{-it}|} dt \left( \frac{1}{2\pi} \int_\T \frac{dt}{|1-ze^{-it}|}\right)^{p-1}. $$
Next, by  Fubini's theorem, we obtain that
\be\label{last:eq}
\int_0^{2\pi} | f_z(re^{i\theta})|^p d\theta \leq |\dot{F}|_p^p \left( \frac{1}{2\pi} \int_\T \frac{dt}{|1-ze^{-it}|}\right)^p.
\ee
When $r$ goes to $1$, we have, see Rudin \cite[Prop 1.4.10]{rud2}
 $$ \frac{1}{2\pi}\int_\T \frac{dt}{|1-ze^{-it}|} \simeq \ln \frac{1}{1-r^2},\quad r=|z|.$$

As $\left( \ln \frac{1}{1-r^2}\right)^p  $ is integrable near $1$, we deduce that $f_z \in L^p(\D)$. We remark that the  identity  (\ref{last:eq}) is not enough to prove that $ f_z$ is in the Hardy space $H^p(\D)$ and the use of M. Riesz's theorem is essential to reach our conclusion. \\

In addition, using Proposition \ref{prop:integ}, we provide pointwise   estimates of $|f_z|$ and $|f_\zb|$ in terms of $\|\dot{F}\|_p$.

\begin{prop}
	Let $f=P[F]$ and $\dot{F}\in L^p(\T)$ with $p\in(1,\infty)$, then there exists $C_p>0$ such that
	$$\max(|f_z(z)|,\, | f_{\zb}(z)|) \leq  \frac{C_p \|\dot{F}\|_p}{(1-|z|^2)^{1/p}} .$$
	
\noindent	Thus
	$$|D_f(z)| \leq  \frac{2C_p \|\dot{F}\|_p }{(1-|z|^2)^{1/p}}, $$
with 
$$C_{p}=\left(\frac{\Gamma(q-1)}{\Gamma^2(q/2)}\right)^{\frac{1}{q}},$$
where $q$ is the conjugate of $p$.
\end{prop}

\begin{proof}
By Proposition \ref{prop:integ}, $f_z$ has the following integral representation    $$
	f_z(z)= \frac{1}{2\pi  i}\int_\T \frac{\dot{F}(w)}{w-z} \, |dw|. 
	$$ Hence,
by H\"older inequality, we deduce that
$$| f_z(z)| \leq \left(\frac{1}{2\pi} \int_\T \frac{dt}{|1-ze^{-it}|^q} \right)^{\frac{1}{q}} ||\dot{F}||_p.$$

Let us denote
$$ I_a(z):=\frac{1}{2\pi} \int_0^{2\pi} \frac{dt}{|1- ze^{-it}|^{a+1}}. $$

In \cite[Proposition 1.1]{KMM}, we proved that if $a>0$, then $$ 	I_{a} (z) \leq\frac{\Gamma(a)}{\Gamma^2(a/2+1/2)} \frac{1}{(1-|z|^2)^a}.$$
Thus
$$|f_z(z)| \leq \left(I_{q-1}(z)\right)^{1/q} \|\dot{F}\|_p\leq \frac{C_p \|\dot{F}\|_p}{(1-|z|^2)^{1/p}}. $$
\end{proof}

As a consequence, we obtain the following about the Lipschitz continuity of harmonic functions on the unit disc $\D$. In  \cite{MMV}, the authors provided a characterizations when a Poisson transformation of a Lipschitz function on the circle is Lipschitz on the disc. 
\begin{prop}
	Let $f=P[F]$ and $\dot{F}\in L^p(\T)$ with $p\in(1,\infty)$. Then $f$ is $\frac{1}{q}$-H\"older continuous on $\D$, where $q$ is the conjugate of $p$.  
\end{prop}

The proof is an immediate consequence of the following proposition essentially due to 	Gehring and Martio.
\begin{prop}\cite{gh}
	Let $f\in\mathcal{C}^1(\D)$ be such that 
	$$|Df(z)|\leq \frac{C}{(1-|z|^2)^{1-\alpha}},\ \ \ \ \ z\in\D,$$
for $0<\alpha\leq1$ and $C>0$. Then $f$ is $\alpha$-H\"older continuous  on $\D$.
\end{prop}




\begin{thebibliography}{99}


\bibitem{CPW} S.L. Chen,  S. Ponnusamy, and   X.T. Wang,: Remarks on ‘Norm estimates of the partial derivatives for harmonic mappings and harmonic quasiregular mappings’. {\it J. Geom. Anal.} {\bf 31}, 11051–11060 (2021)


\bibitem{Du}  P. Duren,
{\it Theory of $H^{p}$ spaces,} 2nd ed., Dover, Mineola, N. Y., 2000.

\bibitem{Du1} P. Duren,
{\it Harmonic mappings in the plane,} Cambridge Univ. Press, 2004.

\bibitem{Gar} J. Garnett, Bounded Analytic Functions, Academic Press, New York, 1981

	\bibitem{gh}	Gehring FW, Martio O., Lipschitz classes and quasiconformal mappings. Ann. Acad. Sci. Fenn. Ser. A I Math {\bf 10} 203–219 (1985)

\bibitem{KMM} Khalfallah, A., Mateljevi\'c, M. and Mhamdi, M. Some Properties of Mappings Admitting General Poisson Representations. {\it Mediterr. J. Math}.  18, {\bf 193} (2021)

\bibitem{Koos} P. Koosis, Introduction to $H^p$ spaces, 2nd ed., Cambridge Tracts in Math. Vol. 115, Cambridge, University Press, Cambridge, UK, 1998


\bibitem{HKZ}  H. Hedenmalm, B. Korenblum and K. Zhu, 
{\it Theory of Bergman spaces}, Springer, New York, 2000.


\bibitem{ver} B. Hollenbeck and I.E. Verbitsky, Best Constants for the Riesz Projection,   {\it J. Fun. Anal.},  {\bf 175} (2),  370-392. (2000)

\bibitem{Mash} J. Mashreghi,   Representation Theorems
in Hardy Spaces, London Mathematical Society Student Texts ({\bf 74}),
Cambridge University Press, 2009

\bibitem{MMV} M. Mateljevic, M. Arsenovic and V. Manojlovic,
Lipschitz-type spaces and  harmonic mappings in the space, {\it Ann. Acad. Sci.Fenn}, Vol. {\bf 35}, No. 2, 2010, 379-387.




\bibitem{Riesz} M. Riesz, Sur les fonctions conjug\'ees, {\it Math. Zeit.}  {\bf 27}, 218-244 (1927)

\bibitem{rud1} W.  Rudin:
	\textit{Real and complex analysis}, McGraw-Hill Book Co, 1966.

\bibitem{rud2} W.  Rudin:, W. {\it Function theory in the unit ball of $\mathbb{C}^n$}, Springer, 1980	


\bibitem{Zhu}  J. F. Zhu,
Norm estimates of the partial derivatives for harmonic mappings and harmonic quasiregular mappings,
\textit{J. Geom. Anal.,}   {\bf 31}, 5505–5525 (2021) 



\end{thebibliography}
\end{document}